\documentclass[10pt,a4paper]{amsart}

\usepackage{caption}
\usepackage{longtable}
\usepackage{graphicx}
\usepackage{epsfig}
\usepackage{color}
\usepackage{array}
\usepackage{amsmath}
\usepackage{amsthm}
\usepackage{amssymb}
\usepackage{enumerate}
\usepackage{amscd}
\usepackage{relsize}
\usepackage[bookmarksnumbered,bookmarksopen,pdfusetitle,unicode]{hyperref}
\usepackage[all,cmtip]{xy}

\numberwithin{equation}{section}

\theoremstyle{plain}
\newtheorem{thm}{Theorem}[section]
\newtheorem{cor}[thm]{Corollary}

\newtheorem{prop}[thm]{Proposition}
\newtheorem{conj}[thm]{Conjecture}

\theoremstyle{definition}
\newtheorem{defn}[thm]{Definition}

\newtheorem{ex}[thm]{Example}

\theoremstyle{remark}
\newtheorem{rem}[thm]{Remark}
\newtheorem{obs}[thm]{Observation}

\graphicspath{ {./figures/} }

\usepackage{subfig}
\renewcommand{\arraystretch}{1.2} 

\title[\L ojasiewicz exponent of a surface: an intrinsic view ]{\ \L ojasiewicz exponent of a surface: an intrinsic view}

\author{E. Bilgin}
\address{Mathematics Group,
Middle East Technical University, 
\newline {\hphantom{a} } Northern Cyprus Campus
99738 Kalkanl{\i}, G{\"u}zelyurt, TRNC, via Mersin 10, Turkey}
 \email{bemel@metu.edu.tr}

 \author{G. Kaya}
\address{Department of Mathematics,  Galatasaray University\\  Ortak{\"o}y 34357, Istanbul, Turkey}
 \email{kagulay@gmail.com}

\author{M. Tosun}
\address{Department of Mathematics, Galatasaray University\\  Ortak{\"o}y 34357,  Istanbul, Turkey}
\email{mrltosun@gmail.com (Corresponding Author)}

\thanks{This work is supported by the projects 113F293 and 18F320 under the programs of the Scientific and Technological Research Council of Turkey. All authors declare that they have no conflicts of interest.}


\begin{document}
\begin{abstract}
For a surface $X$ with an ADE-type singularity, we establish a relation between the elements of the local ring ${\mathcal O}_{X, 0}$ and the \L ojasiewicz exponent $\mathcal{L}_0(X)$ and we give an estimate of $\mathcal{L}_0(X)$ when $X$ has a rational singularity of multiplicity 3.
 \end{abstract}
 
\maketitle \tableofcontents

\section{Introduction}

\noindent  Let $F: \mathbb C^N\longrightarrow \mathbb C$ be an analytic function such that the origin is an isolated singularity of $X=F^{-1}(0)$. The \L ojasiewicz exponent $\mathcal{L}_0(X)$ at ${\bf 0}\in \mathbb C^N$ is defined as the infimum of the elements in the set
\begin{equation}
\left\{\theta>0 \  \mid  \ \ \exists\  U\subset \mathbb C^N\ {\rm and}\  \exists \ c\in \mathbb{R}_+  \ \hbox{\rm such}\  \hbox{\rm that}\  \|{\bf z}\|^{\theta} \leq c\cdot \| \nabla F({\bf z})\|  \   \hbox{\rm for}\  \hbox{\rm all}\ {\bf z}\in U \right\} 
\label{eqn1}
\end{equation}
 Here $\|{\bf z}\|=\max_i\{|z_i|\}$ with ${\bf z}=(z_1,\dots,z_N)\in \mathbb{C}^N$ and   
 $\nabla F=(\frac{\partial F}{\partial z_1},\dots,\frac{\partial F}{\partial z_N}): (\mathbb{C}^N,0)\longrightarrow (\mathbb{C}^N,0)$ is the gradient of $F$. The inequality in \ref{eqn1} is called the \L ojasiewicz gradient inequality.  As the name suggests, the first use of an inequality of this nature is due to \L ojasiewicz in \cite{loj1,loj2}.  It is conjectured in \cite{teis1} that $\mathcal{L}_0(X)$ is a topological invariant. We know that the conjecture is true for the following class of hypersurfaces because the weights are a topological invariant \cite{yau, saeki}. 

\begin{thm}  \label{weight} \cite{KOP} Let $F:\mathbb{C}^N\rightarrow \mathbb{C}$ be a weighted homogeneous polynomial with weight $w=(w_1,\ldots ,w_N)$ and of degree $d$.  Assume $d\geq w_i$ for all $i$. Then 
$$\mathcal{L}_0(F)\leq \max_{i=1}^{N}\{w_i-1\}$$ 
where the equality holds when $N=3$.
\end{thm}

\noindent In particular, if $X$ is a surface with a rational singularity of multiplicity $2$ at the origin, called ADE-type singularity, $\mathcal{L}_0(X)$ can be computed directly by this formula.

{\scriptsize
\renewcommand*{\arraystretch}{1,3}
\begin{table}[h]
    \begin{tabular}{ | p{4.5cm} | p{2.9cm}  |  p{2.5cm}  | p{2.5cm}  |}
    \hline  
ADE-type singularity  & $(w_1,w_2,w_3)$ & $d$ & $\mathcal{L}_0(X)$ \\
 \hline
  $A_n$, $(n=2k)$: $z_3^2+z_2^2+z_1^{n+1}=0$ & $(2,2k+1,2k+1)$ & $4k+2$ & $n$ \\
   \hline
  $A_n$, $(n=2k+1)$: $z_3^2+z_2^2+z_1^{n+1}=0$ & $(1,k+1,k+1)$ & $2k+2$ & $n$ \\
   \hline
  $D_n$: $z_3^2+z_1z_2^2+z_1^{n-1}=0$ & $(2,n-2,n-1)$ & $2(n-1)$ & $n-2$ \\
   \hline
  $E_6$: $z_3^2+z_2^3+z_1^4=0$ & $(3,4,6)$ & $12$ & $3$ \\
   \hline
  $E_7$: $z_3^2+z_2^3+z_1^3z_2=0$ & $(4,6,9)$ & $18$ & $\frac{7}{2}$ \\
   \hline
  $E_8$: $z_3^2+z_2^3+z_1^5=0$ & $(6,10,15)$ & $30$ & $4$ \\
  \hline
 \end{tabular}
  \caption{ \L ojasiewicz exponent of ADE-singularities}
  \label{tablo1}
\label{tab}
\end{table}
}

\noindent  In the next section, we introduce some properties of a rational singularity of multiplicity $m\geq 2$.  In Section 3, we give some nice properties for the \L ojasiewicz exponent of an ideal in the local ring $\mathcal O_{X,0}$. 

\begin{defn}{\rm \cite{lejteis}} \label{ideal} Let $X$ be a reduced equidimensional complex analytic space. 
Let $I:=<f_1,\ldots ,f_r>$ and $J:=<g_1,\ldots ,g_s>$ be two ideals in $\mathcal{O}_X$. The \L ojasiewicz exponent of $I$ with respect to $J$, denoted by $\mathcal{L}_J(I)$, is  the infimum of the set
\begin{equation}
\left\{\eta >0 \mid \  \exists \ U \   {\rm and}\ \exists \ c>0 \ {\rm such}\  {\rm that}\ {\displaystyle \sup_{i=1}^{s}}\  \vert g_i({\bf z})\vert^{\eta } \leq c\  {\displaystyle \sup_{i=1}^{r}}\ \vert f_j({\bf z})\vert \ {\rm for}\ {\rm all  } \ {\bf z}\in U \right\}
\label{eqn2}
\end{equation} where $U$ is an open neighborhood of the origin in $X$. 
\end{defn}

\noindent If this set is empty we say that $\mathcal{L}_J(I)=\infty $. Note that $\mathcal{L}_J(I)< \infty$ when 
$I\subset \sqrt{J}$. 

\begin{prop} {\rm \cite{lejteis}}
With preceding notation, we have $\mathcal{L}_J(I)\in \mathbf Q_+$.
\end{prop}

\noindent In Section 3, we also relate $\mathcal{L}_J(I)$ with the length $\ell (I)$ of the ideal in $\mathcal O_{X,0}$ in order to study 
 the \L ojasiewicz exponent of $X$ by the local data of the singular point  instead of the ambient data in $(\mathbb C^N,0)$.  
 When $X$ has a rational singularity of multiplicity $m\geq 3$ we study the 
\L ojasiewicz exponent of the mapping $F=(f_1,\ldots ,f_k): \mathbb C^N\rightarrow \mathbb C^k$. In Section 4, we compute $\mathcal{L}_0(X)$ when $m=3$ at the origin. We then conjecture that  \L ojasiewicz exponent of a rational singularity of multiplicity $m$ is bounded by the length of a special ideal in $\mathcal O_{X,0}$.  

\vskip.2cm

\section{Integrally closed ideals in the local ring of a rational singularity}

\noindent Let $X$ be the germ of a normal surface in $\mathbb C^N$ with a rational singularity at $0$ and  ${\mathcal O}_{X,0}$ be its local ring. An element $g\in {\mathcal O}_{X,0}$ is said to be integral on an ideal $I\subset  {\mathcal O}_{X,0}$ if it satisfies an equation $g^n +a_1g^{n-1} +\ldots +a_n =0$
with $a_i\in I^i$ for all $i = 1,\ldots ,n$. Denote by $\bar I$ the set of all elements in ${\mathcal O}_{X,0}$ which are integral  over $I$; it is an ideal and called the integral closure of $I$ in  ${\mathcal O}_{X,0}$. We have
$I\subseteq \bar I$. When $\bar I=I$ we say that $I$ is an integrally closed ideal in ${\mathcal O}_{X,0}$ \cite{Swanson-Huneke}. 


\vskip.2cm

\noindent   Let $\pi: \tilde X\rightarrow X$ the minimal resolution of $X$. It is well known 
that $E:=\pi^{-1}(0)$ is normal crossing and each irreducible component $E_i$ of $E=E_1\cup \ldots \cup E_n$ is a rational curve. 

\begin{thm}{\rm \cite{Lipman}} \label{lip} The product of integrally closed ideals in ${\mathcal O}_{X,0}$ is an integrally closed ideal in ${\mathcal O}_{X,0}$. 
\end{thm}

\noindent Let ${\mathcal M}$ be the maximal ideal  in ${\mathcal O}_{X,0}$. An ideal $I$ in ${\mathcal O}_{X,0}$ is called ${\mathcal M}$-primary if ${\mathcal M}=\sqrt I$ and $\bar I\subseteq \sqrt I$.  Let $\mathcal{S}({\bf I})$ be the set of ${\mathcal M}$-primary integrally closed ideals $I$ in ${\mathcal O}_{X,0}$ such that the pullback $I\mathcal{O}_{\tilde X}$ of $I$ by $\pi $ is invertible; equivalently, ${I}{\mathcal O}_{\tilde X}={\mathcal O}(-D_I)$ where $D$ is a positive divisor supported on $E$. The set $\mathcal{S}({\bf I})$ is a semigroup with respect to the product of ideals. The elements in 
$\mathcal{S}({\bf I})$ can be studied using their associated positive divisors as follows: By \cite{Lipman}, any element $h$ in ${\mathcal O}_{X,0}$ defines a positive divisor $D_h$ supported on $E$ such that $\pi^{\ast }(h)=D_h+T_h$ where $T_h$ is the strict transform of $h$ by $\pi $. If we denote by $\nu_{E_i}(h)$ the vanishing order of the divisor $\pi^{\ast }(h)$ along $E_i$, we have $D_h=\sum_{i=1}^{n}\nu_{E_i}(h)E_i$ with $\nu_{E_i}(h)\geq 1$ for all $i$ since $E$ is connected. Hence $(D_h\cdot E_i)\leq 0$ for all $i$. Let $\mathcal{E}(\pi )$ be the set of the positive divisors $D_h$ supported on $E$ for all $h\in {\mathcal O}_{X,0}$. The set $\mathcal{E}(\pi )$ is a semigroup with respect to the addition, called {\it the semigroup of Lipman associated with} $\pi $. 

\begin{thm}{\rm \cite{Artin, Lipman}}
\label{bijection}
There exists a bijection between the semigroups $\mathcal{S}({\bf I})$ and $\mathcal{E}(\pi )$. 
\end{thm}

\noindent  In fact, for  $I, J\in \mathcal{S}({\bf I})$ we have 
$IJ{\mathcal O}_{\tilde X}={\mathcal O}_{\tilde X}(-D_I-D_J)$ with $D_I+D_J\in \mathcal{E}(\pi )$. 
Conversely,  to each positive divisor $D$ supported on $E$ such that $D\cdot E_i\leq 0$ for all $i$, 
we associate un ideal $I_D$ in ${\mathcal O}_{X,0}$ defined as the stalk at $0$ of $\pi_{\ast }{\mathcal O}_{\tilde X}(-D)$. 

\begin{defn} \cite{bondil} Let $I\in \mathcal{S}({\bf I})$. An element $g\in I$ is called generic in $I$ if $\nu_{E_i}(g)\leq\nu_{E_i}(h)$ for all $h\in I$ and  for all $i=1,\ldots ,n$.
\end{defn}

\noindent  The order $\nu_{E_i}(I)$ of an ideal $I$ is defined as 
$$\nu_{E_i}(I):=inf\lbrace \nu_{E_i}(h)\mid h\in I\rbrace .$$ Hence, for a generic element $g$ in $I$, we have 
$\nu_{E_i}(I)=\nu_{E_i}(g)$, so $D_I=D_g$. This says that any ideal $I$ in $\mathcal{S}({\bf I})$ corresponds to an element $D_I\in \mathcal{E}(\pi )$ through its generic element. The smallest element in $\mathcal{E}(\pi )$ is called the Artin divisor $D_{\mathcal M}$ of $\pi $ and corresponds to (generic element of) the  maximal ideal ${\mathcal M}$ in ${\mathcal O}_{X,0}$. Moreover, $\mathcal{E}(\pi )$ is a partially ordered  set by $\leq $ defined as follows: For any $D,D'\in \mathcal E(\pi )$ we say that $D\leq D'$ whenever $\nu_{E_i}(D)\leq \nu_{E_i}(D')$ for all $i=1,\ldots ,n$; here the notation $\nu_{E_i}(D)$ means $\nu_{E_i}(h)$ such that $h$ is the generic element in the ideal $I_D$. To compute a generating set for the semigroup $\mathcal{E}(\pi )$, let us fix an ordering on the irreducible components of $E$. Let ${M}(E)=(e_{ij})$ denote the intersection matrix of $E$ where $e_{ij}=(E_i\cdot E_j)$ for all ${1\leq i,j\leq n}$; it is a negative definite symmetric matrix. If $h$ is the generic element of an ideal $I\in \mathcal{S}({\bf I})$ its associated divisor $D_h=\sum_{i=1}^{n}\nu_{E_i}(h)E_i$ in $\mathcal{E}(\pi )$ satisfies the equality:
$${M}(E)\cdot(\nu_{E_1}(h), \nu_{E_2}(h),\ldots, \nu_{E_n}(h))^t=(y_1,y_2,\ldots,y_n)^t$$
where $y_i\leq 0$ for all $i$ except at least one $y_{i_0}<0$ where $i_0\in \lbrace 1,\ldots ,n\rbrace $. Let us denote by $\delta_i$ the column matrix with coefficients $0$ everywhere except in the $i$-th row, where the entry is $-1$. Consider the system ${M}(E)\cdot (m_{i1},m_{i2},\cdots ,m_{in})^t=\delta_i$. 
Put $F'_i=\sum_{j=1}^{n}m_{ij}E_j$. We have $m_{ij}\in \mathbf Q^+$ for all $i,j$. Write $F'_i$ as $F_i=k_i\cdot {F'_i}$ for (smallest) $k_i\in  \mathbf Q^+$ 
such that the coordinates of the $F_i$ are positive integers and relatively prime. By construction, each ${F_i}$ belongs to ${\mathcal E}(\pi )$ and each element in $\mathcal{E}(\pi )$ can be written as a linear combination of $F_1,\ldots ,F_n$ with coefficients in 
$\mathbf Q^+$. Put ${\mathcal G}(\pi  ):=\lbrace F_1,\ldots ,F_n\rbrace $.  An element of ${\mathcal G}(\pi  )$ is called a {\it $\mathbf Q$-generator} 
for ${\mathcal E}(\pi  )$. 

\section{\L ojasiewicz exponent and length of an ideal}

\noindent  Consider two ideals $I,J\in {\mathcal O}_{X,0}$. By \cite{lejteis, bivia}, the \L ojasiewicz exponent of 
$I$ with respect to $J$ is given by 
\begin{equation}
 \mathcal{L}_J(I)={\min}\lbrace \frac{a}{b}\mid a,b\in {\mathbf Z}_{\geq 1},\  \  {J}^a\subseteq {\overline {I^b}}\rbrace 
 \label{eqn3}
 \end{equation}
 In particular, we have:

\begin{prop} 
\label{tan} Let $I,J\in \mathcal{S}({\bf I})$. 
Then  
$$\mathcal{L}_J(I):={\min}\lbrace \frac{a}{b}\mid a,b\in {\mathbf Z}_{\geq 1},\  \  a\cdot D_{J}\geq b\cdot D_I\rbrace $$
\end{prop} 

\begin{proof}
Put ${\overline {I^b}}=I^b$ since $I\in \mathcal{S}({\bf I})$ is integrally closed. Then we rewrite the inclusion at the right hand side in terms of the associated divisors through their generic elements. Since $I\cdot {\mathcal O}_{X}$ is locally principal in $\tilde X$ we have $\mathcal{O}_{\tilde X}(-D_{J})\supseteq \mathcal{O}_{\tilde X}(-D_I)$; so $D_I\geq D_{J}$. The fact $D_{I^b}=b\cdot D_I$ gives the inequality. 
\end{proof}


\begin{prop} {\rm \cite{teis2}}
\label{defn}
Let $I\in \mathcal{S}({\bf I})$. Then 
$$ \mathcal{L}_0(I)={\displaystyle \max_{i=1}^n} \left\{ \frac{\nu_{E_i}(D_I)}{\nu_{E_i}(D_{\mathcal M})}\right\} $$
 In particular, we have $\mathcal{L}_0(\mathcal{M})=1$.
\end{prop}

\begin{cor} Let $I,J\in \mathcal{S}({\bf I})$. Then we have $\mathcal{L}_J(I)\geq \frac{\mathcal{L}_0(I)}{\mathcal{L}_0(J)}$.
\end{cor}

\begin{proof} Let $\mathcal{L}_0(J)=\displaystyle\max_{i=1}^n \left\{\frac{\nu_{E_i}(D_J)}{\nu_{E_i}(D_{\mathcal M})}\right\}=\frac{\nu_{E_k}(D_J)}{\nu_{E_k}(D_{\mathcal M})}$. 
In this case \\

$ {\mathlarger { \frac{\mathcal{L}_0(I)}{\mathcal{L}_0(J)}}}=\displaystyle \frac{\displaystyle \max_{i=1}^n  \left\{ \frac{\nu_{E_i}(D_I)}{\nu_{E_i}(D_{\mathcal M})}\right\}} { \frac{\nu_{E_k}(D_J)}{\nu_{E_k}(D_{\mathcal M})}} = \displaystyle  \max_{i=1}^n  \left\{ \frac{\nu_{E_k}(D_{\mathcal M})\nu_{E_i}(D_I)}{\nu_{E_k}(D_J)\nu_{E_i}(D_{\mathcal M})}\right\}   \leq \max_{i=1}^n \left\{ \frac{\nu_{E_i}(D_I)}{\nu_{E_i}(D_J)}\right\} = \mathcal{L}_J(I) $.
\end{proof}


\begin{prop}  For $I,J\in \mathcal{S}({\bf I})$ we have
$$\mathcal{L}_J(I)\geq \frac{(D_I\cdot D_I)}{(D_I\cdot D_J)}$$
\end{prop}

\begin{proof}
 By \cite{leteis, bondil}, the number $-(D_I\cdot D_I)$ equals the multiplicity $e(I)$ of the ideal $I$ in $\mathcal O_{X,0}$. By \cite{teis2}, we have $\mathcal{L}_J(I)\geq \frac{e(I)}{e_1(I\mid J)}$. Here $e_1(I\mid J)$ represents the mixed multiplicity of the ideals $I,J\in \mathcal{S}({\bf I})$ which is  given by
$$e_1(I\mid J)=\frac{1}{2} (e(IJ)-e(I)-e(J)) $$
Since $e(IJ)=-(D_I+D_J)\cdot
(D_I+D_J)$,  we have $e_1(I\mid J)=-(D_I\cdot D_J)$. \end{proof}

\begin{prop} For $I_1,I_2,J\in \mathcal{S}({\bf I})$, we have
$$\mathcal{L}_J(I_1I_2)=\mathcal{L}_J(I_1)+\mathcal{L}_J(I_2)$$
\end{prop}

\begin{proof}
 From the previous discussion we have:
\begin{eqnarray*}
\mathcal{L}_J(I_1I_2) & = & \max_{i=1}^{n}\frac{\nu_{E_i}(D_{I_1I_2})}{\nu_{E_i}(D_J)} \\
 & = &  \max_{i=1}^{n}\frac{\nu_{E_i}(D_{I_1} + D_{I_2})}{\nu_{E_i}(D_J)}\\
 & = &  \max_{i=1}^{n}\frac{\nu_{E_i}(D_{I_1}) + \nu_{E_i}(D_{I_2}) }{\nu_{E_i}(D_J)}\\
 & = &  \max_{i=1}^{n}\frac{\nu_{E_i}(D_{I_1})}{\nu_{E_i}(D_J)} + \max_{i=1}^{n}\frac{\nu_{E_i}(D_{I_2})}{\nu_{E_i}(D_J)}\\
 & = &  \mathcal{L}_J(I_1)+\mathcal{L}_J(I_2)
\end{eqnarray*}\end{proof}

\begin{cor} \label{power} For any $k\in \mathbf N^{\ast }$ and $I\in \mathcal{S}({\bf I})$ we have
$$\mathcal{L}_0(I^k)=k\cdot \mathcal{L}_0(I)$$
\end{cor}

\begin{rem} \label{exist}
When $X$ is an ADE-type singularity, Table \ref{tablo2} shows that there exists an element $I\in \mathcal{S}({\bf I})$ such that $\mathcal{L}_0(X)=\mathcal{L}_0(I)$. For an $A_n$-type singularity, using Theorem \ref{weight}, we get $\mathcal{L}_0(A_n)=n$. Note that 
the Artin divisor $D_{\mathcal M}$ is reduced, means $\nu_{E_i}(D_{\mathcal M})=1$ for all $i$; so, by Proposition \ref{defn}, the biggest coefficient $\nu_{E_i}(D_{\mathcal M})$ in each $D_I\in \mathcal E(\pi )$ gives the  \L ojasiewicz exponent of the corresponding $I\in \mathcal{S}({\bf I})$. In particular, $\mathcal{L}_0(I)=n$ for the ideal $I$ corresponding to the $\mathbf Q$-generator $D_I=E_1+2E_2+\ldots +(n-1)E_{n-1}+nE_n$ of $\mathcal{E}(\pi )$. For simplicity, we will use the notation $D_I=(\nu_{E_1}(D_I), \nu_{E_2}(D_I),\ldots ,\nu_{E_n}(D_I))=(1,2,3,\ldots ,n-1,n)$. 

\end{rem}

\noindent {\bf Length of an ideal.} The length (or co-length) of an ideal $I$ in a ring $R$, denoted by $\ell (R/I)$, is the dimension of $R/I$ over the field $\bf k$. Since $X$ has a rational singularity, each ${\mathcal M}$-primary ideal in ${\mathcal O}_{X,0}$ has finite length and the length of $I$, will be denoted by $\ell (I)$, in ${\mathcal O}_{X,0}$ is $dim_{\mathbb C}({\mathcal O}_{X,0}/I)$. Obviously, $\ell ({\mathcal M})=1$. It is easier to compute the length of an 
ideal in $\mathcal{S}({\bf I})$ using its associated divisor in  ${\mathcal E}(\pi )$.

\begin{thm} {\rm (\cite{watanabe}, Remark 3.2)} \label{div}
Let $I \in  \mathcal{S}({\bf I})$. Then
$${\ell }(I)=\frac{-(D_I\cdot D_I)-\sum_{i=1}^{n}\nu_{E_i}(D_I)(w_i-2)}{2}$$
where $w_i=-E_i^2$ for all $i$.
\end{thm}

\noindent In the sequel, we use the notations ${\ell }(D_I)$ and ${\ell }(I)$ equivalently and call the length of the ideal $I$. 

\vskip.2cm

\begin{prop}
With preceding notation,  we have $\mathcal{L}_{\mathcal M}(I)\leq \ell (I)$ for every $I\in \mathcal{S}({\bf I})$. 
\end{prop}

\begin{proof}
It results from the fact that we have $\mathcal{M}^p\subseteq I$ if ${\ell }(I)=p$ for an ideal $I$ (see Tables \ref{tablo2}).
\end{proof}

\begin{rem}  The length defines the map $\ell :\mathcal{S}({\bf I})\longrightarrow \mathbb{R}$ and we have $\ell (I\cdot J)=\ell (D_I)+\ell (D_J)$. For each $p\in \mathbb N^{\ast }$, there may not exist an ideal of length $p$ in $\mathcal O_{X,0}$ and, if exists, there are a finite number of ideal of length $p$ (see Table \ref{tablo2}).
\end{rem}

\noindent Let $D_I$ be an element in ${\mathcal E}(\pi )$. Consider the components $E_i$ of $E$ such that $(D_I\cdot E_i)<0$. Let us reindex these components as $F_1,\ldots ,F_k$. We have $k\leq n$ where $n$ is the total number of the irreducible components of $E$. Consider the set 
$$E-\lbrace F_1,\ldots ,F_k\rbrace =\prod {\mathcal E}^{j}$$
with $j=1,\ldots ,s$. Each sub-configuration ${\mathcal E}^{j}$ is called Tjurina component of $E$ with respect to $D_I$ and all elements of ${\mathcal E}(\pi )$ can be constructed by one of the process given in the following theorem: 

\begin{prop} \cite{tosun} Let $I,J \in  \mathcal{S}({\bf I})$. Then we have $D_I=D_J+D'$ for some positive divisor $D'$.   

\noindent $(i)$ If $D'=Z({\mathcal E}^j)$ we have $\ell (I)=\ell (J)+1$,

\noindent $(ii)$ If $D'=E_{i_0}$ such that $E_{i_0}$ is attached only to the vertices $E_j$ of $E$ such that $(D_I\cdot E_{j})<0$, we have $l(I)=\ell (J)-(D_J\cdot D')+1$.
\end{prop}

\begin{proof} Consider two divisors $D_I$ and $D_J$ in ${\mathcal E}(\pi )$ such that there exists a sequence of divisors in 
${\mathcal E}(\pi )$ such that $D_I=D_{I_0}<D_{I_1}<\ldots <D_{I_l}=D_J$; this corresponds to the sequence of ideals
$J=I_l\subset \ldots \subset I_1\subset I_0=I$ in $\mathcal{S}({\bf I})$.  Note that both sequences are not uniquely defined and 
If  $\ell (D_I)-\ell (D_J)=1$ for some $i$ then $D_{I_i}$ and $D_{I_{i+1}}$ are adjacent. 
By \cite{altinok-tosun}, for each $j$, the divisor $D_I+Z({{\mathcal E}^j})$ is in ${\mathcal E}(\pi )$ where $Z({\mathcal E}^j)$ is the smallest divisor of ${\mathcal E}^j$. This gives $(i)$. If $(D_I\cdot E_{i_0})<0$  for a vertex $E_{i_0}$ of $E$ and $E_{i_0}$ is attached only to the vertices $E_j$ of $E$ such that $(D_I\cdot E_{j})<0$ then the divisor $D_I+E_{i_0}$ is  in ${\mathcal E}(\pi )$. 
\end{proof}

\begin{rem} Let $I, J_1, J_2\in \mathcal{S}({\bf I})$. If $J_1\subseteq J_2$ then $D_{J_1}\geq D_{J_2}$. So $\ell (J_1)\geq \ell (J_2)$ and $\mathcal{L}_I(J_1)\geq \mathcal{L}_I(J_2)$.
However, if $J_1$ is not contained in the increasing sequence of $J_2$ then we
may get $\ell (J_1)\geq \ell (J_2)$ with $\mathcal{L}_I(J_1) <
\mathcal{L}_I(J_2)$. In particular, when $\ell (J_1) = \ell (J_2)$ we may
get $\mathcal{L}_I(J_1) \neq \mathcal{L}_I(J_2)$ as Table \ref{tablo2} shows.
\end{rem}

\begin{ex}
\noindent Consider the case where $X$ is an $E_8$-type singularity. The integral closure of the Jacobian ideal ${\mathcal J}=<z_3,z_2^2,z_1^4>$ is $\overline{{\mathcal J}}= (z_1^4,z_2^2,z_1^2z_2,z_3)$. Note that it is hard to compute the integral closure of an ideal if it is not a monomial ideal. The length $\ell (\overline{{\mathcal J}})$ is $6$. We have $\overline{{\mathcal J}}\subset {\mathcal M}$, ${\mathcal M}^5\subset \overline{{\mathcal J}}$ but 
${\mathcal M}^4\not \subset \overline{{\mathcal J}}$. More precisely, we get  ${\mathcal M}^{4t+1}\subset \overline{{\mathcal J}}^{t}$, ($t\in \mathbb N^*$), so $\mathcal{L}_0({\mathcal J})=4$. Furthermore,  by  \cite{gs, snoussi}, we know that  the generic element ${\mathfrak p}$ of $\overline {\mathcal J}$ corresponds to the divisor $D_{\mathfrak p}=(5,10,15,12,9,6,3,8)$ in ${\mathcal E}(\pi )$  according to the ordering in $E$ taken as  
 $$\begin{array}{ccccccc} E_1 & E_2 & E_3 & E_4 & E_5 & E_6 & E_7 \\  &  &  E_8 &  &   &  &  \end{array}$$
 We obtain $\ell (\overline{{\mathcal J}})=\ell (D_{\mathfrak p})=4$.  As Table \ref{tablo2} shows, we have $\mathcal{L}_0(X)\leq \ell (D_{\mathfrak p})$ for E-type singularities.


{\tiny 
\renewcommand*{\arraystretch}{1,3}
\begin{table}[h]
    \begin{tabular}{ | p{2.5cm} | p{.4cm}  | p{.7cm}  || p{2.8cm} | p{.4cm}  | p{.7cm} || p{3.5cm} | p{.4cm} | p{.7cm} |}
    \hline  
 Some $D_I$'s in $\mathcal{E}(\pi )$ for $E_6$-type &  $\ell (I)$  & $\mathcal{L}_0(I)$ &  Some $D_I$'s in $\mathcal{E}(\pi )$ for {\bf $E_7$}-type & $\ell (I)$ &  $\mathcal{L}_0(I)$  &  Some $D_I$'s in $\mathcal{E}(\pi )$ for {\bf $E_8$}-type & $\ell (I)$ &  $\mathcal{L}_0(I)$  \\
 \hline
 $(1,2,3,2,1,2)*$      & 1 & 1           &  $(2,3,4,3,2,1,2)*$    & 1  & 1            &   $(2,4,6,5,4,3,2,3)*$      &  1 & 1 \\ 
\hline
  $(2,3,4,3,2,2)$       & 2 & 2           &  $(2,4,6,5,4,2,3)*$    & 2  & 2            &   $(4,7,10,8,6,4,2,5)*$    & 2 & 2    \\
\hline
$D_{\mathfrak{p}}=(2,4,6,4,2,3)*$        & 3 & 2            &  $(2,4,6,5,4,3,3)*$    &  3 & 3/2         &  $(4,8,12,10,8,6,3,6)*$    & 3 & 2  \\
\hline
 $(4,5,6,4,2,3)*$     & $6 $ & 4       &    $(3,6,8,6,4,2,4)*$   &   3 & 2          &  $D_{\mathfrak{p}}=(5,10,15,12,9,6,3,8)*$   & 4 &  8/3   \\
\hline
 $(2,4,6,5,4,3)*$      & $6 $ & 4        & $D_{\mathfrak{p}}=(3,6,9,7,5,3,5)$  &  4 & 3 & $(6,12,18,15,12,8,4,9)*$    &  6 & 3   \\
\hline 
 $(3,6,8,6,3,4)$       & 6 & \textcolor{red}{3}           &  $(4,8,12,9,6,3,6)*$  &  6 & 3         &  $(7,14,20,16,12,8,4,10)*$  &   7 & 7/2 \\
\hline 
 $(3,6,9,6,3,5)$       & 7 & 3          & $(4,8,12,9,6,3,7)*$         & $7 $ & \textcolor{red}{7/2}       & $(7,14,21,17,13,9,5,11)$   & $8 $  & 11/3 \\
\hline 
 $(5,10,12,8,4,6)*$  &  15   &    4                & $(5,9,12,9,6,3,6)$         & $7 $ & 3       & $(8,14,20,16,12,8,4)$   & 8  & \textcolor{red}{4} \\
 \hline
  $(4,8,12,10,5,6)*$  &   15  &     4           &  $(4,8,12,10,7,4,6)$  &  7   & 4                & $(8,16,24,20,15,10,5,12)*$    & $10$ & 4 \\
\hline  
                                 &    &             &     $(6,12,18,15,10,5,9)*$  &  15   & 5                      &   $(10,20,30,24,18,12,6,15)*$   & 15 & 5 \\       
 \hline                                          
 \end{tabular}
 \caption{The $\mathbf Q$-generators are represented by * in each case}
 \label{tablo2}
\end{table}
}
\end{ex}

\begin{rem} Let $X=V(I)$. The generic element ${\mathfrak p}$ of the integral closure of the ideal $\mathcal J+I$ defines a curve, called the polar curve of $X$ \cite{teis1}. We have $\pi^{\ast }({\mathfrak p})=D_{\mathfrak p}+T_{\mathfrak p}$ and  the strict transform $T_{\mathfrak p}$ of ${\mathfrak p}$ by $\pi $ intersects the irreducible components $E_i$'s with $(D_{\mathfrak p}\cdot E_i)<0$.  These intersection points of $T_{\mathfrak p}$ and $E$ give the base points of ${\mathfrak p}$. In the case of ADE-singularities and the rational singularities with reduced Artin cycle,  the base points of ${\mathfrak p}$ are described precisely in \cite{gs, snoussi, spivakovsky}. It is still an open problem for other classes of rational singularities. Here we relate $\ell (D_{\mathfrak p})$ and $\mathcal{L}_0(X)$ and,  for RTP-singularities,  we find the candidate divisors for $D_{\mathfrak p}$.  
\end{rem}

\noindent  When $X$ is an $A_n$-type singularity, using \cite{spivakovsky},  we know that the strict transform $T_{\mathfrak p}$ of the polar curve  passes through the intersection point of two irreducible components in the middle of $E$ when with $n=2k$ and, $T_{\mathfrak p}$ intersects the unique irreducible component which is in the middle of $E$ when $n2k+1$. Hence 
$$D_{\mathfrak p}=(1,2,3,\ldots ,k-1,k,k,k-1,\ldots ,3,2,1), \ \ D_{\mathfrak p}=(1,2,3,\ldots,k,k+1,k,\ldots ,3,2,1)$$ for $n=2k$ and $n=2k+1$ respectively. Here the ordering of the irreducible components of $E$ is taken as 
$$\begin{array}{ccccccc} E_1 & E_2 & E_3 & E_4 & \ldots  & E_{n-1} & E_{n}  \end{array}$$ 
When $X$ is a $D_n$-type singularity, using  \cite{snoussi, gs} we get: 
$$D_{\mathfrak p}=(k,2k,2k-1,\ldots ,2,1,k), \ \ D_{\mathfrak p}=(k,2k,2k-1,2k-2,\ldots ,4,3,2,k), (k\geq 2)$$ 
for $n=2k$ and $n=2k+1$ respectively such that the ordering in $E$ is taken as 
$$\begin{array}{ccccccc} 
E_1 & E_2 & E_3 & \ldots & E_{n-1} \\  
&  E_n &  &  &   \end{array}$$
Consequently, when $X$ is $A_n$-type or $D_n$-type, we have ${\ell }(D_{\mathfrak p})=\frac{n}{2}$ for $n$ even, ${\ell }(D_{\mathfrak p})=\frac{n+1}{2}$ for $n$ odd. 

\noindent  In the cases of $E_6$-type and $E_7$-type singularities,  we get the  divisors $D_{\mathfrak p}$ as given in Table \ref{tablo2} 
with respect to the following orderings in $E$ 
$$\begin{array}{ccccccccccc} 
E_1 & E_2 & E_3 & E_4 & E_5 \ \  and    &  E_1 & E_2 & E_3 & E_4 & E_5 & E_6\\  
&   & E_6 &  &   &  &   & E_7 &  & &  \end{array}$$

\begin{obs}  If $X$ is of $ADE$-type then 
\begin{equation}
\label{inegalite}
\mathcal{L}_0(X)\leq \frac{mult_0(X)}{mult_0(X)-1}\cdot \mathcal{L}_0({D_{\mathfrak p}})
\end{equation}
\noindent We can also replace $\mathcal{L}_0({D_{\mathfrak p}})$ in the inequality by $\ell ({D_{\mathfrak p}})$.
\end{obs}

\vskip.2cm

\noindent Recall  that, for ADE-singularities, the Milnor number  $\mu (X)=dim_{\mathbb C}(\frac{\mathcal O_{\mathbb C^N,0}}{\mathcal J})$ of $X$ equals  the Tjurina number $\tau (X)=dim_{\mathbb C}(\frac{\mathcal O_{\mathbb C^N,0}}{<f,\frac{\partial f}{\partial z_1},\frac{\partial f}{\partial z_2},\frac{\partial f}{\partial z_3}>})$ which is the dimension of the base space of a semi-universal deformation of $X$ \cite{saito, wahl}. 

\section{\L ojasiewicz exponent of rational singularities of higher multiplicity}

\noindent Assume that $X\subset \mathbb C^N$ is the germ of a surface with a rational singularity at the origin of multiplicity $m>2$. By \cite{Artin},  the multiplicity $m$ equals $N-1$ and, by \cite{wahl}, $X$ is defined by $q:=\frac{(N-1)(N-2)}{2}$ equations. In other words, there exists $q$ germs of holomorphic functions  $f_i:\mathbb C^N\rightarrow \mathbb C$ so that the fiber over $0$  of the application $F:(f_1,f_2,\ldots ,f_q):\mathbb C^N\rightarrow \mathbb C^q$ is the surface $X=\{{\bf z}\in \mathbb C^N\mid  f_1({\bf z})=\ldots =f_q({\bf z})=0\}$ with multiplicity $N-1$ at $0$.  Let $g_1,\dots,g_s$ be the determinants of $(N-2)\times (N-2)$ minors of the Jacobian matrix $(\frac{\partial f_i}{\partial z_j})_{i,j}$ where $s:=\binom{N}{N-2}\binom{q}{N-2}$. The ideal ${\mathcal J}=<g_1,\dots,g_s>$ is called the Jacobian ideal. Consider the map
$$G:=(g_1,\dots,g_s):{\mathbb C}^N\rightarrow {\mathbb C}^{s}$$ 
As in the case where $q=1$, the  \L ojasiewicz exponent  of $X$ is defined as  the smallest element of the set 
\begin{equation}
\left\{\theta>0 \  \mid  \ \ \exists\  U\subset \mathbb C^N\ {\rm and}\  \exists \ c\in \mathbb{R}_+  \ \hbox{\rm such}\  \hbox{\rm that}\  \|{\bf z}\|^{\theta} \leq c\cdot \| G({\bf z})\|  \   \hbox{\rm for}\  \hbox{\rm all}\ {\bf z}\in U \right\} 
\label{eqn4}
\end{equation}

\noindent that is, $\mathcal{L}_0(X)=\mathcal{L}_0({\mathcal J})$.  Now let us restrict our attention on a special class of rational singularities. For this, recall that a map $F=(f_1,\ldots ,f_r):{\mathbb C}^N\longrightarrow {\mathbb C}^r$ is called quasi-homogeneous if there exists  ${\bf w}\in (\mathbf R_+-\lbrace 0\rbrace )^N$ and ${\bf d}\in (\mathbf R_+-\lbrace 0\rbrace )^r$ 
such that, for each $i$, we have
$$f_i(\lambda^{w_1}z_1, \lambda^{w_2}z_2,\ldots , \lambda^{w_N}z_N)= \lambda^{d_i}f_i(z_1,z_2,\ldots ,z_N)$$
\noindent where ${\bf w} = (w_1,\ldots ,w_N)$ and ${\bf d}=(d_1,d_2,\ldots ,d_r)=(d(f_1),\ldots ,d(f_r))$.

\begin{ex} 
\label{harau} Let $F=(f_1,f_2,f_3):{\mathbb C}^4\longrightarrow {\mathbb C}^3$ defines a rational singularity of multiplicity 3, called RTP-singularities. By  \cite{tj,mag}, they are defined by the equations presented in the following table, so each of them is quasi-homogeneous.
\end{ex} 
{\tiny
\renewcommand*{\arraystretch}{0,5}
\begin{table}[h]
    \begin{tabular}{ | p{1.9cm} | p{2.9cm} | | p{1.7cm} | p{2.9cm} | | p{1.4cm} | p{2.9cm} | }
    \hline
    {\bf RTP-type } & {\bf Equations} & {\bf RTP-type} &  {\bf Equations} & {\bf RTP-type} &  {\bf Equations}  \\
    \hline

$A_{k-1,\ell-1,m-1}$ & $xw-y^{m}w-y^{\ell+m}=0$ & $C_{k-1,\ell+1}$ & $xz-y^{k}w=0$ & $D_{k-1}$ & $xz-y^{k+2}-y^kw=0$ \\  $k\geq\ell \geq m\geq 1$  & $zw+y^{\ell}z-y^{k}w=0$    &   $k\geq 1$, $\ell \geq 2$    & $w^2-x^{\ell+1}-xy^2=0$  & $k\geq 1$  	   & $zw-x^2y^k=0$  \\ 
                                             & $xz-y^{m+k}=0$                 &                     			    & $zw-x^\ell y^k-y^{k+2}=0$  & & $w^2+y^2w-x^3=0$ \\
  \hline
$B_{k-1,n}$ & $xz-y^{k+\ell}-y^{k}w=0$ & $B_{k-1,n}$ & $xz-y^kw=0$ & $F_{k-1}$  &   $xz-y^kw =0$ \\ 
 $n=2\ell>3$  &  $w^2+y^\ell w-x^2y=0$ & $n=2\ell-1\geq 3$  & $zw-xy^{k+1}-y^{k+\ell}=0$ & $k\geq 1$ &$zw-x^2y^k-y^{k+3} =0$\\ 
        		       	   &     $zw-xy^{k+1}=0$                    &    				         &   $w^2-x^2y-xy^\ell=0$ &     			&   $ w^2-x^3-xy^3=0$ \\ 
           \hline            
$H_{n}$      	 &   $z^2-xw=0$  	 &   $H_{n}$  &  $z^2-xy^{k+1}-xyw=0$ & $H_{n}$ 	&  $z^2-xw=0$ 			\\ 
       $n=3k$       &   $zw+y^kz-x^2y=0$     	&   $n=3k+1$       &  $zw-x^2y=0$ & $n=3k-1$   &  $zw-x^2y-xy^k=0$ \\ 
                          &   $w^2+y^kw-xyz=0$    	&                       &  $w^2+y^kw-xz=0$ &            &  $w^2-y^kz-xyz=0$	  \\ 
           \hline    

       $E_{6,0}$		& $z^2-yw=0$ &  $E_{0,7}$    & $z^2-yw=0$                 &  $E_{7,0}$  &  $z^2-yw=0$ \\  
       			        &  $zw+y^2z-x^2y=0$ & 	& $zw-x^2y-y^4=0$         &                    &  $zw+x^2z-y^3=0$            \\ 
                              & $w^2+y^2w-x^2z=0$ & & $w^2-x^2z-y^3z=0$      &                   &   $w^2+x^2w-y^2z=0$       \\ 
           \hline    
    \end{tabular}
    \caption{The equations defining RTP-singularities}
    \label{tablo3}
\end{table}
}

\begin{prop} \cite{haraux} If $X$ is an RTP-singularity then $G=(g_1,\dots,g_{18}):{\mathbb C}^4\rightarrow {\mathbb C}^{18}$ is quasi-homogeneous with weight  ${\bf w}=(w_1,w_2,w_3,w_4)$ and with quasi-degree  ${\bf d}=(d(g_1),\ldots d({g_{18}}))\in \mathbf Z^{18}_{>0}$ and the \L ojasiewicz exponent $\mathcal{L}_0(X)$ is bounded as
$$ \frac{\min_{i,j}\lbrace d(g_1),\ldots d({g_{18}})\rbrace }{\min\lbrace w_1,w_2,w_3,w_4\rbrace}\leq  
\mathcal{L}_0(X)\leq  \frac{\max\lbrace d(g_1),\ldots d({g_{18}})\rbrace }{\min\lbrace w_1,w_2,w_3,w_4\rbrace}.$$
\end{prop} 

\noindent The possible values of the  \L ojasiewicz exponent of $\mathcal{L}_0(X)$ can be computed using the explicit equations and represented in the table below.

{\tiny
\renewcommand*{\arraystretch}{1,1}
\begin{table}[h]
    \begin{tabular}{ | p{1.0cm} | p{3.9cm} | p{0.5cm} | p{2.4cm} | p{3.6cm}  | } 
    \hline  
    {\bf RTP}       &  $\  \  {\bf weights}$  & min$\bf w$  &  min$\bf d$     &  max$\bf d$  \\
    \hline
$A_{k,\ell ,m}$ & $(m,1,k,\ell )$  &  1 &  $2m$   &        $2k+\ell -1$      \\
\hline
$B_{2\ell }$  &   $(2\ell -1,2,2k+1,2\ell )$ &  2  &   $4\ell -1$ for $l<k+1$, $4k+2$ for $l\geq k+1$  &    $4k+2\ell $ for $k=\ell $,  $4k+\ell $ for $l<k$,   $\ \ $ $6\ell $ for $l\geq k+1$   \\
\hline
$B_{2\ell -1}$ & $(2\ell -2, 2, 2k + 1, 2\ell -1)$ for $\ell < k+1$, $(2k,2,k+\ell ,k+\ell )$ for $\ell \geq k+1$ & 2 &  $4k+1$ for $\ell \geq k+1$,  $4\ell -3$ for $\ell <k+1$ &  $3k+3\ell -2$ for $\ell \geq k+1$, $2k+2\ell +4$  for $\ell <k$, 
$2k\ell -\ell -1$ for $\ell =k$  \\
\hline
$C_{k-1,\ell+1}$  & $(2,\ell ,k.\ell +\ell -1, \ell +1)$  & 2 &  $\ell +3$   &  $2k\ell +2\ell -2$  \\
\hline
$D_{k-1}$             &  $(4,3,3k+2,6)$     &  3 &     10  &  $6k+4$      \\
\hline
$F_{k-1}$             & $(6,4,4k+3,9)$     &  4 &   15       & $8k+6$     \\
\hline
 $H_{3k-1}$         &    $(3k-3, 3, 3k-2, 3k-1)$  & 3   &   $6k-4$  & $8,14,12k-12$ for $k\geq 4$     \\
\hline
$H_{3k}$             &    $(3k-2,3,3k-1,3k)$    & 3 &  $ 6k-2$  &  $18,27, 8k+4$ for $k\geq 4$   \\
\hline
 $H_{3k+1}$        &    $(3k-1, 3, 3k + 1, 3k)$  &  3  &   $6k-4$     &  $19,29,12k-4$ for $k\geq 4$    \\
\hline
$E_{6,0}$           &   (5,4,6,8)            &  4   &    13         &   24           \\
\hline
$E_{0,7}$           &   (9, 6, 10, 14)      &  6  &    20        &   36        \\ 
\hline
$E_{7,0}$           &  (5, 6, 8, 10)       &  5    &    16      &   30      \\ 
\hline
 \end{tabular}
\caption{$\mathcal{L}_0(X)$ for RTP-singularities}
\label{tablo4}
\end{table}
}

\noindent The following theorem gives a nice upper bound on $\mathcal{L}_0(X)$ by showing that the upper bound we obtained for the ADE-singularities is also valid for the RTP-singularities. 

\begin{thm}
\label{sinir}
Let $X$ be a surface with an RTP-type singularity.  Then the inequality \ref{inegalite} holds.
\begin{equation}
\mathcal{L}_0(X)\leq \frac{mult_0(X)}{mult_0(X)-1}\cdot \mathcal{L}_0({D_{\mathfrak p}})
\end{equation}
\end{thm}

\begin{proof} \label{olmadi}
Let $X$ be of $E_{0,7}$-type singularity.  We have $X=V(I)$ with $I=<f_1,f_2,f_3>$. Consider the ideal 
$${\mathcal J}=<J_{ij}(z_1z_2), J_{ij}(z_1z_3), J_{ij}(z_1z_4), J_{ij}(z_2z_3), J_{ij}(z_2z_4), J_{ij}(z_3z_4)>$$
with $i,j=1,2,3$ and $i\neq j$ where $J_{ij}(z_1z_2)=\frac{\partial f_i}{\partial z_1}\frac{\partial f_j}{\partial z_2}-\frac{\partial f_i}{\partial z_2}\frac{\partial f_j}{\partial z_1}$. A generating set for ${\mathcal J}$  is 
\begin{align*}
  \begin{array}{ccccccccc}
 & 2z^2+yw, & 6x^2z+17w^2, & xzw, & xyw, & x^2w, & xyz, &  &  \\
 3x^2y-17zw, & 19y^3w+4x^2w, & y^3z+x^2z+2w^2, & y^4+x^2y-4zw, & x^4 & xy^2 & w^2 & yzw & \\
  \end{array}
\end{align*}

\noindent The length of the ideal ${\mathcal J}$ is $17$. It is hard to compute the integral closure of that ideal. However, we can still make use of the formula given in \ref{eqn3} in order to get an estimation on $\mathcal{L}_0({\mathcal J})$.   If for some $a,b\in \mathbb{Z}_{\geq 1}$ we have ${\mathcal M}^a\subseteq {\mathcal J}^b$, then this implies that $\mathcal{L}_0({\mathcal J})\leq \frac{a}{b}$. We obtain ${\mathcal M}^{5}\subset {\mathcal J}$ and the best estimation we get is the inclusion ${\mathcal M}^{4t+1}\subset {\mathcal J}^{t}$ with $t\in \mathbb N^*$, so we can say $\mathcal{L}_0(X)=\mathcal{L}_0( {\mathcal J})\leq 4$. We also conclude that ${D_{\mathfrak p}}$ is among the divisors with 
length $\leq 17$. In Table \ref{tablo5}, we give all possible $\mathcal{L}_0(D)$ for the divisors $D$ with lengths $\leq 17$  and we see that the inequality \ref{sinir} holds.

\noindent Now, consider the surface $X\subset \mathbb{C}^4$ of  $E_{6,0}$-type singularity. A generating set of the Jacobian ideal ${\mathcal J}$  is 
 \begin{align*}
 \begin{array}{ccccccc}
   w^2, & zw, & y^2w-2w^2, & xyw, & x^2w-4yzw, & z^3, & xz^2,  \\
     xyz, & 2x^2z-4yz^2-3w^2, & y^3+2z^2+yw,  & xy^2, &  x^2y-2y^2z-4zw, & y^2z^2 & x^4\\
   \end{array}
\end{align*}
\noindent  The length $\ell ({\mathcal J})$ equals $16$. Again we proceed with $\mathcal J$ instead of the integral closure of $\mathcal J$ and we obtain ${\mathcal M}^5\subseteq {\mathcal J}$ but ${\mathcal M}^4\nsubseteq {\mathcal J}$. More precisely, we have ${\mathcal M}^{8t}\subset {\mathcal J}^{2t}$  with $t\in \mathbb N^*$. Hence $\mathcal{L}_0(X)=\mathcal{L}_0( {\mathcal J})\leq 4$ and we again have the inequality.

\newpage 

{\tiny
\begin{table}
\centering 
   \begin{tabular}{ | p{2.7cm} | p{0.6cm}  | p{0.6cm}  | | p{3.4cm} | p{0.6cm}  | p{0.6cm}  | }
\hline
Some elements in $\mathcal E(\pi )$ &  $ l(I)$ &  $\pounds_0(I)$ & Some elements in $\mathcal E(\pi )$ &  $ l(I)$ &  $\pounds_0(I)$ \\  [0.5ex]  
\hline 
$(1, 3, 4, 3, 2,1, 2)*$        &  1 &  1  & $(1,3,5,4,3,2,1,3)*$ &          1 & 1\\
  \hline
$(2, 3, 4, 3, 2,1, 2)*$        &  2 &  2  & $(2,4,6,5,4,3,2,3)$ &          2 & 2 \\
  \hline
$(2, 4, 6, 5, 4,2, 3)$          &  3 &  2  & $(2,5,8,7,6,4,2,4)$ &           3 & 2 \\
  \hline
$(2, 4, 6, 5, 4,3, 3)$          &  4 &  3  & $(2,6,10,8,6,4,2,5)*$ &          4 & 2 \\
  \hline
$(2, 5, 8, 6, 4,2, 4)$          &  4 &  2  & $(2,5,8,7,6,5,3,4)$ &           4 & 3 \\
  \hline
$(2, 5, 8, 7, 5,3, 4)$          &  5 &  3  & $(3,7,10,8,6,4,2,5)$ &          5 & 3 \\
  \hline
$(2, 6,10, 9, 6,3, 5)*$        &  8 &  3  & $(2,6,10,9,8,6,3,5)$ &         6 & 3 \\
  \hline
$(2, 6,10, 8, 6,4, 5)$         &  8 &  4  & $(2,5,8,7,6,5,4,4)$ &          6 & 4 \\
  \hline
$(2, 6,10, 8, 6,3, 6)$         &  8 &  3  & $(3,8,12,10,8,6,3,6)$ &         7 & 3\\
  \hline
$(3, 7,11, 9, 6,3, 6)$         &  9 &  3  & $D_{\mathfrak{p}}=(4,7,10,8,6,4,2,5)*$ &    8 & 4 \\
  \hline
$(2, 5, 8, 7, 6,5, 4)$          & 10 &  5  & $(2,6,10,9,8,7,6,5)*$ &        14 & 6 \\
  \hline
$(2, 6,10, 9, 8,4, 5)*$       & 11 &  4  & $(4,10,16,14,12,10,5,8)$   &          17 & 5\\
  \hline
 $(3, 8,13,11, 8,4, 7)$       & 13 &  4  & $(4,10,16,14,12,9,6,8)$    &       17 & 6 \\
    \hline
$(3, 9,14,12, 9,5, 7)$       & 16 &  5  &   $(4,10,16,14,12,10,5,8)$     &        17 & 5 \\
  \hline
$(3, 9,14,11, 8,5, 7)$       & 16 &  4  & $(4,11,18,15,12,9,5,10)$     &         17 & 5 \\
  \hline
$(3, 9,15,12,8,4, 8)$        & 16 &  6  &  $(5,11,17,14,11,8,5,9)$         &        17 & 5\\
  \hline
$(4,10,14,11,8,4,7)$        & 16 &  4  &  $(5,12,18,15,12,8,4,9)$        &      17    & 5  \\
  \hline
$(4, 9,14,12, 8,4, 7)$       & 16 & 4  & $(4,12,19,16,12,8,4,10)$      &        17 & 5  \\
  \hline
$(4,9,14,11,8,5,7)$          & 16 &  5 & $(4,12,18,15,12,9,5,9)$         &        17 & 5 \\
  \hline
$(2, 6,10, 9, 8,7, 5)*$       & 20 &  7  & $(7,21,30,24,18,12,6,15)*$ &    49 & 7 \\
  \hline
$(4,12,20,15,10,5,10)*$    & 28&5 & $(6,18,30,27,24,16,8,15)*$ &    57 & 8 \\
  \hline
$(4,12,20,15,10,5,13)*$.   &37&13/2&  $(8,24,40,36,27,18,9,20)*$ &    86 & 9 \\ [1ex] 
  \hline                           
\end{tabular}
\caption{For $E_{6,0}$ and $E_{0,7}$ singularities} 
\label{tablo5} 
\end{table}
}

\noindent In a similar way, we obtain the following bounds on  $\mathcal{L}_0(X)$ when $X$ is of other RTP-type singularities.

{\tiny
\renewcommand*{\arraystretch}{1,3}
\begin{table}[h]
    \begin{tabular}{ | p{1.0cm} | p{5.4cm}  | p{2.9cm} | }
    \hline  
{\bf RTP}                 &   $u$ with $\ell ({\mathcal J})\leq u$   &  $j$ with $\mathcal{L}_0({\mathcal J})\leq j$\\
\hline
$A_{k,\ell ,m}$        &   $k+\ell +m+5$           & $k+\ell $ with $k\geq \ell \geq m\geq 1$  \\
\hline
$B_{k,2\ell }$          &   $5\ell +k+2$   for $l\geq k+1$, $k+4l+2$ for $l<k+1$  &   $k+\ell +1$  \\
\hline
$B_{k,2\ell -1}$       &   $5\ell +k$ for $\ell \leq k$, $5\ell +k-1$ for $\ell \geq k+1$ & $k+\ell $ \\
\hline
$C_{k-1,\ell+1}$      &   $k+\ell +8$   &   $k+2$ \\
\hline
$D_{k-1}$                &   $k+13$       &  $k+3$  \\
\hline
$F_{k-1}$                &   $k+16$        &  $k+3$  \\
\hline
 $H_{3k-1}$             &   $6k+2$    &  $2k$ \\
\hline
$H_{3k}$                 &   $6k+4$     &   $2k$  \\
\hline
 $H_{3k+1}$            &    $6k+6$      &   $2k+1$ \\
\hline
$E_{6,0}$                &      $16$           &   4 \\
\hline
$E_{0,7}$                &     $17$            &   5    \\ 
\hline
$E_{7,0}$                &     $17$             &   5   \\ 
\hline
 \end{tabular}
\caption{$\ell ({\mathcal J})$ and $\mathcal{L}_0({\mathcal J})$ for RTP-singularities respectively}
\label{table}
\end{table}
}

\end{proof}

\noindent The computation above gives a nice upper bound on $\mathcal{L}_0(X)$ which permits also to determine the approximate location of $D_{\mathfrak{p}}$.

\begin{conj}
The inequality in Theorem \ref{sinir} is true for a rational singularity of higher multiplicity.
\end{conj}

\vskip.2cm

\noindent {\bf Funding and/or Conflicts of interests/Competing interest:} The authors have no relevant financial or non-financial interests to disclose.

\vskip.2cm


\end{document}